\documentclass[12pt,twoside]{amsart}
\usepackage{amsmath, amsthm, amssymb}
\usepackage{paralist}  
\usepackage{manfnt}    
\usepackage{url}
\usepackage[all]{xy}
\CompileMatrices

\evensidemargin \oddsidemargin
\newtheorem{thm}{Theorem}[section]
\newtheorem{theorem}[thm]{Theorem}
\newtheorem*{nnthm}{Theorem}

\newtheorem{cor}[thm]{Corollary}

\newtheorem{lemma}[thm]{Lemma}

\newtheorem{prop}[thm]{Proposition}

\theoremstyle{definition}

\theoremstyle{remark}
\newtheorem{remark}[thm]{Remark}

\newtheorem{example}[thm]{Example}

\newcommand{\ds}{\displaystyle}

\title[Embedding Orders into Central Simple Algebras] {Embedding
  Orders into Central Simple Algebras}

\author {Benjamin Linowitz}

\address{Department of Mathematics\\ 
6188 Kemeny Hall\\
Dartmouth College\\
Hanover, NH 03755}

\email[] {benjamin.linowitz@dartmouth.edu}
\urladdr{http://www.math.dartmouth.edu/~linowitz/ }

\author {Thomas R. Shemanske}

\address{Department of Mathematics\\ 
6188 Kemeny Hall\\
Dartmouth College\\
Hanover, NH 03755}

\email[] {thomas.r.shemanske@dartmouth.edu}
\urladdr{http://www.math.dartmouth.edu/~trs/ }

\date{\today}

\newcommand{\B}{{\mathcal B}}
\newcommand{\D}{{\mathcal D}}
\newcommand{\E}{{\mathcal E}}
\newcommand{\f}{{\mathfrak f_{\Omega/\O_K}}}
\newcommand{\ff}{{\mathfrak f}}
\renewcommand{\L}{{\mathcal L}}
\newcommand{\M}{{\mathcal M}}
\newcommand{\fM}{{\mathfrak M}}

\newcommand{\N}{{\mathcal N}}

\newcommand{\fN}{{\mathfrak N}}

\renewcommand{\O}{{\mathcal O}}
\newcommand{\fP}{{\mathfrak P}}
\newcommand{\hL}{{\widehat L}}
\newcommand{\Q}{\mathbb{Q}}
\newcommand{\R}{{\mathcal R}}
\newcommand{\Z}{\mathbb{Z}}

\newcommand{\End}{\mathrm {End}}

\newcommand{\Frob}{\mathrm {Frob}}
\newcommand{\gen}{\mathrm {gen}}
\newcommand{\ord}{\mathrm {ord}}
\newcommand{\diag}{\mathrm {diag}}
\newcommand{\la}{\langle}
\newcommand{\ra}{\rangle}

\parskip=\medskipamount

\begin{document}
\subjclass[2010] {Primary 11R54; Secondary 11S45, 20E42}

\keywords{Order, central simple algebra, affine building, embedding}

\begin{abstract}
  The question of embedding fields into central simple algebras $B$
  over a number field $K$ was the realm of class field theory.  The
  subject of embedding orders contained in the ring of integers of
  maximal subfields $L$ of such an algebra into orders in that algebra
  is more nuanced.  The first such result along those lines is an
  elegant result of Chevalley \cite{Chevalley-book} which says that
  with $B = M_n(K)$ the ratio of the number of isomorphism classes of
  maximal orders in $B$ into which the ring of integers of $L$ can be
  embedded (to the total number of classes) is $[L \cap \widetilde K :
  K]^{-1}$ where $\widetilde K$ is the Hilbert class field of $K$.
  Chinburg and Friedman (\cite{Chinburg-Friedman}) consider arbitrary
  quadratic orders in quaternion algebras satisfying the Eichler
  condition, and Arenas-Carmona \cite{Arenas-Carmona} considers
  embeddings of the ring of integers into maximal orders in a broad
  class of higher rank central simple algebras.  In this paper, we
  consider central simple algebras of dimension $p^2$, $p$ an odd
  prime, and we show that arbitrary commutative orders in a degree $p$
  extension of $K$, embed into none, all or exactly one out of $p$
  isomorphism classes of maximal orders.  Those commutative orders
  which are selective in this sense are explicitly characterized;
  class fields play a pivotal role.  A crucial ingredient of Chinberg
  and Friedman's argument was the structure of the tree of maximal
  orders for $SL_2$ over a local field.  In this work, we generalize
  Chinburg and Friedman's results replacing the tree by the
  Bruhat-Tits building for $SL_p$.
\end{abstract}

\maketitle


\section{Introduction}\label{sec:introduction} The subject of
embedding fields and their orders into a central simple algebra
defined over a number field has been a focus of interest for at least
80 years, going back to fundamental questions of class field theory
surrounding the proof of the Albert-Brauer-Hasse-Noether theorem as
well as work of Chevalley on matrix algebras.

To place the results of this paper in context, we offer a brief
historical perspective.  A major achievement of class field theory was
the classification of central simple algebras defined over a number
field, and the Albert-Brauer-Hasse-Noether theorem played a pivotal
role in that endeavor.  For quaternion algebras, this famous theorem
can be stated as:

\begin{nnthm}
  Let $B$ be a quaternion algebra over a number field $K$, and let
  $L/K$ be a quadratic extension of $K$.  Then there is an embedding
  of $L/K$ into $B$ if and only if no prime of $K$ which ramifies in
  $B$ splits in $L$.
\end{nnthm}

The quaternion case is fairly straightforward to understand since a
quaternion algebra over a field is either $2\times 2$ matrices over
the field or a (central simple) division algebra.  The field extension
$L/K$ is necessarily Galois, so the term splits is unambiguous.

In the general setting, we have a central simple algebra $B$ of
dimension $n^2$ over a number field $K$. From \cite{Pierce-book} p
236, $L/K$ embeds into $B$ only if $[L:K] \mid n$, and an embeddable
extension of degree $n$ is called a strictly maximal extension.  The
theorem above generalizes as follows.  For a number field $K$, and
$\nu$ any prime of $K$ (finite or infinite), let $K_\nu$ be the
completion with respect to $\nu$ and let $B_\nu =B \otimes_K K_\nu$ be
the local central simple algebra of dimension $n^2$ over $K_\nu$.  The
Wedderburn structure theorem says that $B_\nu \cong
M_{\kappa_\nu}(D_\nu)$ where $D_\nu$ is a central simple division
algebra of dimension $m_\nu^2$ over $K_\nu$, so of course $n^2 =
\kappa_\nu^2 m_\nu^2$.  We say that the algebra $B$ ramifies at $\nu$
iff $m_\nu > 1$, and is split otherwise.  The generalization of the
classical theorem above follows from Theorem 32.15 of
\cite{Reiner-book} and the corollary on p 241 of \cite{Pierce-book}.

\begin{nnthm}
  Let the notation be as above, and suppose that $[L:K] = n$.  Then
  there is an embedding of $L/K$ into $B$ if and only if for each
  prime $\nu$ of $K$ and for all primes $\fP$ of $L$ lying above
  $\nu$, $m_\nu \mid [L_\fP:K_\nu]$.
\end{nnthm}

For example, any extension $L/K$ of degree $n$ will embed in $M_n(K)$
as $m_\nu = 1$ for all $\nu$.  So now we turn to the question of
embedding orders into central simple algebras which is considerably
more subtle.  Perhaps the first important result was due to Chevalley
\cite{Chevalley-book}.

Let $K$ be a number field, $B = M_n(K)$, $L/K$ a field extension of
degree $n$ and we may assume (without loss of generality from above)
that $L \subset B$. Let $\O_L$ be the ring of integers of $L$.  We
know (see p 131 of \cite{Reiner-book}) that $\O_L$ is contained in
some maximal order $\R$ of $B$, but not necessarily all maximal orders
in $B$.  Chevalley's elegant result is:

\begin{nnthm}
  The ratio of the number of isomorphism classes of maximal orders in
  $B$ into which $\O_L$ can be embedded to the total number of
  isomorphism classes of maximal orders is $[\widetilde K\cap L:
  K]^{-1}$ where $\widetilde K$ is the Hilbert class field of $K$.
\end{nnthm}

In the last decade or so, there have been a number of generalizations
of Chevalley's result.  In 1999, Chinburg and Friedman
\cite{Chinburg-Friedman} considered general quaternion algebras
(satisfying the Eichler condition), but arbitrary orders $\Omega$ in
the ring of integers of an embedded quadratic extension of the center,
and proved a ratio of 1/2 or 1 with respect to maximal orders in the
algebra (though the answer is not as simple as Chevalley's).  Chan and
Xu \cite{Chan-XU}, and independently Guo and Qin \cite{Guo-Qin}, again
considered the quaternion algebras, but replaced maximal orders with
Eichler orders of arbitrary level.  Maclachlan \cite{Maclachlan}
considered Eichler orders of square-free level, but replaced
embeddings into Eichler orders with optimal embeddings. The first
author of this paper \cite{Linowitz-selectivity} replaced Eichler
orders with a broad class of Bass orders and considered both
embeddings and optimal embeddings.

The first work beyond Chevalley's in the non-quaternion setting was by
Arenas-Carmona \cite{Arenas-Carmona}.  The setting was a central
simple algebra $B$ over a number field $K$ of dimension $n^2$, $n\ge
3$ with the proviso that the completions of $B$ (at the
non-archimedean primes) have the form (in the notation above) $B_\nu
\cong M_{\kappa_\nu}(D_\nu)$, $n = \kappa_\nu m_\nu$, with $\kappa_\nu
= 1$ or $n$.  He considered embeddings of the ring of integers $\O_L$
of an extension $L/K$ of degree $n$ into maximal orders of $B$ and
proves a result analogous to Chevalley's with the Hilbert class field
replaced by a spinor class field.  His results come out of the theory
of quadratic forms, in particular from his generalization of the
notion of a spinor class field to the setting of a skew-Hermitian
space over a quaternion algebra.

In this paper, we too consider generalizations beyond the quaternionic
case.  We consider the case where $B$ is a central simple algebra
having dimension $p^2$ ($p$ be an odd prime) over a number field $K$;
this part of the setup is of course a special case of the one in
\cite{Arenas-Carmona}.  On the other hand, like Chinburg and Friedman,
we are able to describe the embedding situation for all orders $\Omega
\subset \O_L$ where $L/K$ is a degree $p$ extension.  As in the case
of Chinburg and Friedman, the question of the proportion of the
isomorphism classes of maximal orders into which $\Omega$ embeds is
not simply dependent on a class field as in the case of the maximal
order $\O_L$, but also on the relative discriminant (or conductor) of
$\Omega$, and it is these considerations which have constrained our
consideration to algebras of degree $p^2$.

There are other substantive differences between \cite{Arenas-Carmona}
and this work.  Generalizing the ideas of \cite{Chinburg-Friedman}, we
are able to parametrize the isomorphism classes of maximal orders in
the algebra so as to give an explicit description of those maximal
orders into which $\Omega$ can be embedded, explicit enough to specify
them via the local-global correspondence.  Also as in
\cite{Chinburg-Friedman} we are able to define the notion of a
``distance ideal'' associated to two maximal orders. We use this
distance ideal together with the Artin map associated to $L/K$ to
characterize the isomorphism classes of maximal orders into which
$\Omega$ can be embedded.

Central to the arguments of Chinburg and Friedman are properties of
the tree of maximal orders over a local field (the Bruhat-Tits
building for $SL_2$).  This paper avails itself to the structure of
the affine building for $SL_p$, but introduces new arguments to
replace those where the quaternionic case utilized the structure of
the building as a tree; smaller accommodations are required since the
extension $L/K$ need not be Galois as it is in the quadratic case.

One interesting observation about all the generalizations mentioned
above is that class fields have played a central role.  We now
describe the main result.  Since the question of embeddability of
fields has been answered above, we presume throughout that $L/K$ is a
degree $p$ extension and that $L \subset B$.  Let $\O_K$ denote the
ring of integers of $K$, and let $\Omega$ denote a commutative
$\O_K$-order of rank $p$ in $L$, so necessarily $\Omega$ is an
integral domain with field of fractions equal to $L$. It follows that
$\Omega$ is contained in a maximal order $\R$ of $B$ (see p 131 of
\cite{Reiner-book}), so we fix $\R$ for the remainder of this paper.
Finally, we define the conductor of $\Omega$ as $\f = \{x\in \O_L \mid
x\O_L \subset \Omega\}$ (see \cite{Neukirch-book}).

Via class field theory, we associate an abelian extension $K(\R)/K$ to
our maximal order $\R$. We find that $\Omega$ embeds into all of the
isomorphism classes of maximal orders except when the following two
conditions are satisfied:
\begin{compactenum}
\item $L \subseteq K(\R)$,
\item Every prime ideal $\nu$ of $K$ which divides $N_{L/K}(\f)$
  splits in $L/K$.
\end{compactenum}
When these two conditions hold, $\Omega$ embeds in one- $p$th of the
isomorphism classes of maximal orders, and those classes are
characterized explicitly by means of the Frobenius, $\Frob_{L/K} \in
Gal(L/K)$.

Following \cite{Chinburg-Friedman}, an order $\Omega \subset \R$, but
which does not embed in all maximal orders is called selective.  In
section~\ref{sec:selective_orders}, we give examples and show that a
degree $p$ division algebra admits no selective orders.

\section{Local Results}\label{sec:local_results} We begin with some
results about orders in matrix algebras over local fields which we
will need.  Let $k$ be a non-archimedean local field, with unique
maximal order $\O$, $V$ an $n$-dimensional vector space over $k$, and
identify $\End_k(V)$ with $\B = M_n(k)$, the central simple matrix
algebra over $k$.  The ring $M_n(\O)$ is a maximal order in $\B$ and
can be denoted as the endomorphism ring $\End_\O(\L)$, where $\L$ is
an $\O$-lattice in $V$ of rank $n$. It is well known ((17.3) of
\cite{Reiner-book}) that every maximal order in $\B$ has the form $u
M_n(\O) u^{-1} = \End_\O(u\L)$ for some $u \in \B^\times$, and is it
trivial to check the for another $\O$-lattice $\M$, we have
$\End_\O(\L) = \End_\O(\M)$ iff $\L$ and $\M$ are homothetic: $\L =
\lambda \M$ for some $\lambda \in k^\times$.

It is also the case that the maximal orders in $\B$ are in one-to-one
correspondence with the vertices of the affine building associated to
$SL_n(k)$ (see \S 6.9 of \cite{Abramenko-Brown}, or Chapter 19 of
\cite{Garrett}), and so the vertices may be labeled by homothety
classes of lattices in $V$, see p 148 of \cite{Brown}.  To realize
such a labeling it is convenient to choose a basis $\{\omega_1, \dots,
\omega_n\}$ of $V$.  This basis, actually the lines spanned by the
basis elements, determines an apartment, and each vertex in that
apartment can be identified uniquely with the homothety class of a
lattice of the form $\O \pi^{a_1}\omega_1 \oplus \cdots \oplus
\O\pi^{a_n}\omega_n$, where $\pi$ is the local uniformizer of $k$.
Since the basis and uniformizer are fixed, we shall denote this
homothety class simply by $[a_1, \dots, a_n]$, ($(a_1, \dots, a_n) \in
\Z^n/\Z(1,1,\dots, 1)$).

Let $\fM_1, \fM_2$ be two maximal orders in $\B=M_n(k)$, and write $\fM_i
= \End_\O(\L_i)$ ($i=1,2$) for $\O$-lattices $\L_i$ in $V$. Since
$\End_\O(\L_i)$ does not depend upon the homothety class of $\L_i$, we
may assume without loss that $\L_1 \subseteq \L_2$.  As the lattices
are both free modules over a PID, they have well-defined invariant
factors: $\{\L_2 : \L_1\} = \{ \pi^{a_1}, \dots, \pi^{a_n}\}$, with
$a_i \in \Z$, and $a_1 \le \cdots \le a_n$. Note that $\{\L_1: \L_2\}
= \{ \pi^{-a_n}, \dots, \pi^{-a_1}\}$.  Define the `type distance'
between $\fM_1$ and $\fM_2$ via the $\L_i$ to be congruence class modulo $n$:
\[
td_k(\fM_1,\fM_2) = td_\pi(\fM_1,\fM_2) \equiv \sum_{i=1}^n a_i \pmod n, \textrm{ where } \{\L_2 :
\L_1\} = \{ \pi^{a_1}, \dots, \pi^{a_n}\}.
\]
This definition depends only on the local field, not the choice of
uniformizer.  The motivation for this definition comes from a
consideration of how to label the vertices of a building.  Those in
the building for $SL_n(k)$ have types 0, \dots, $n-1$.  Any given
vertex, say the one corresponding to the homothety class of $\L$, can be
assigned type 0.  Then if $\alpha \in GL_n(k) = \B^\times$, the vertex
associated to the homothety class of $\alpha \L$ has type congruent to
$\ord_\pi(\det \alpha) \pmod n$ (see \cite{Ronan}).

\section{Maximal Orders over Number Fields}
In returning to the global setting, we recall that we are assuming
that $p$ is an odd prime, and $B$ is a central simple algebra having
dimension $p^2$ over a number field $K$.  For a prime $\nu$ of $K$, we
denote by $K_\nu$ its completion at $\nu$ and for $\nu$ a finite
prime, $\O_\nu$ the maximal order of $K_\nu$, and $\pi = \pi_\nu$ a
fixed uniformizer.  We will denote by $J_K$ the idele group of $K$ and by
$J_B$ the idele group of $B$.  We denote by $nr$ the
reduced norm in numerous contexts: $nr: B \to K$, $nr: B_\nu
\to K_\nu$, or $nr: J_B \to J_K$, with any possible ambiguity resolved
by context.

Because the degree of $B$ over $K$ is odd, $B_\nu \cong M_p(K_\nu)$
for every infinite prime $\nu$ of $K$, and since $p$ is prime, for any
finite prime $\nu$ of $K$, $B_\nu$ is either $M_p(K_\nu)$ ($\nu$ is
said to split in $B$) or a central simple division algebra over
$K_\nu$ ($\nu$ is said to ramify in $B$) (see section 32 of
\cite{Reiner-book}).

Given a maximal order $\R \subset B$, and a prime $\nu$ of $K$, we
define localizations $\R_\nu \subset B_\nu$ by:

\[
\R_\nu =
\begin{cases}
  \R \otimes_\O \O_\nu&\textrm{if $\nu$ is finite}\\
  \R \otimes_\O K_\nu = B_\nu&\textrm{if $\nu$ is infinite}\\
\end{cases}
\]

We will also be interested in the normalizers of the local orders, as
well as their reduced norms.  Let $\N(\R_\nu)$ denote the normalizer
of $\R_\nu$ in $B_\nu^\times$.  When $\nu$ is an infinite prime,
$\N(\R_\nu) = B_\nu^\times$ and $nr(\N(\R_\nu)) = K_\nu^\times$.  If
$\nu$ is finite, we have two cases: If $\nu$ splits in $B$, then
$B_\nu \cong M_p(K_\nu)$ and every maximal order is conjugate by an
element of $B_\nu^\times$ to $M_p(\O_\nu)$, so every normalizer is
conjugate to $GL_p(\O_\nu) K_\nu^\times$ (37.26 of
\cite{Reiner-book}), while if $\nu$ ramifies in $B$, $\R_\nu$ is the
unique maximal order of the division algebra $B_\nu$
\cite{Reiner-book}, so $\N(\R_\nu) = B_\nu^\times$.  It follows that
for $\nu$ split, $nr(\N(\R_\nu)) = \O_\nu^\times (K_\nu^\times)^p$,
while for $\nu$ ramified p 153 of \cite{Reiner-book} gives that
$nr(\N(\R_\nu)) = nr(B_\nu^\times) = K_\nu^\times.$

\subsection{Type Numbers of Maximal Orders}

We say that two orders $\R$ and $\E$ in $B$ are in the same genus if
$\R_\nu \cong \E_\nu$ for all (finite) primes $\nu$ of $K$.  By the
Skolem-Noether theorem, this means they are locally conjugate at all
finite primes.  Denote by $\gen(\R)$ the genus
of $\R$, the set of orders $\E$ in $B$ which are in the same
genus as $\R$.  Again by Skolem-Noether, $\gen(\R)$ is the disjoint
union of isomorphism classes.  The type number of $\R$, $t(\R)$,
is the number of isomorphism classes in $\gen(\R)$.

By Theorem 17.3 of \cite{Reiner-book}), any two maximal orders in $B$
are everywhere locally conjugate, so the genus of maximal orders is
independent of the choice of representative.  So if $\R$ is any
maximal order in $B$, the type number of $\R$ is simply the number of
isomorphism classes of maximal orders in $B$.  The question we answer
is into how many of the isomorphism classes of maximal orders can an
order $\Omega$ be embedded?  Notice that if $\Omega$ embeds into one
maximal order in an isomorphism class it embeds into all, since any
two elements of an isomorphism class are (globally) conjugate.

Adelically, the genus of an order $\R$ is characterized by the coset space
$J_B/\fN(\R)$, where $\fN(\R) = J_B \cap \prod_\nu \N(\R_\nu)$ where
$\N(\R_\nu)$ is the normalizer of $\R_\nu$ in $B_\nu^\times$.  The type number
of $\R$ is the cardinality of the double coset space $B^\times\backslash
J_B/\fN(\R)$.  To make use of class field theory, we need to realize
this quotient in terms of the arithmetic of $K$.  

Henceforth, let $\R$ be a maximal order in $B$. We prove 

\begin{theorem}\label{thm:type_number_bijection}The reduced norm on
  $B$ induces a bijection 
\[
nr: B^\times\backslash J_B/\fN(\R) \to K^\times \backslash J_K/nr(\fN(\R)).
\]  
\end{theorem}
\begin{proof}The map is defined in the obvious way with
$nr(B^\times \tilde \alpha\, \fN(\R)) = K^\times nr(\tilde \alpha)\
nr(\fN(\R))$ and where $nr((\alpha_\nu)) = (nr(\alpha_\nu))$.
  We observed above that no infinite prime of $K$
  ramifies in $B$, so it follows from Theorem 33.4
  of \cite{Reiner-book}, that $nr(B_\nu^\times) = K_\nu^\times$ for all primes of
  $K$, including the infinite ones.  Let $\tilde a = (a_\nu)
  \in J_K$, and $K^\times \tilde a\ nr(\fN(\R))$ be an element of
  $K^\times \backslash J_K/nr(\fN(\R))$.  We construct an idele
  $\tilde \beta = (\beta_\nu) \in J_B$ so that $B^\times \tilde
  \beta \ \fN(\R) \mapsto K^\times \tilde a\, nr(\fN(\R))$. For
  all but finitely many non-archimedean primes $\nu$ of $K$,
  $a_\nu\in \O_\nu^\times$ and $\R_\nu \cong M_p(\O_\nu)$.  Define
  $\beta_\nu$ to be the conjugate of the diagonal matrix
  $\diag(a_\nu, 1, \dots, 1)$ which is contained in $\R_\nu$.
For the other primes, using the local surjectivity of the reduced norm
described above, let $\beta_\nu$ be any preimage of in $B_\nu^\times$ of
$a_\nu$.  The constructed element $\tilde \beta$ is trivially
seen to be in $J_B$ and given the invariance of the reduced norm under
conjugation, we see that $nr(\tilde \beta) = \tilde a$ which
establishes surjectivity.

For injectivity we first need a small claim: that the preimage of $
K^\times nr(\fN(\R))$ under the reduced norm is $B^\times J_B^1
\fN(\R)$, where $J_B^1$ is the kernel of the reduced norm map $nr:J_B
\to J_K$.  It is obvious that $B^\times J_B^1 \fN(\R)$ is contained in
the kernel.  Let $\tilde \gamma \in J_B$ be such that $nr(B^\times
\tilde \gamma\ \fN(\R)) \in K^\times nr(\fN(\R))$. Then
$nr(\tilde\gamma) \in K^\times nr(\fN(\R))$, so write
$nr(\tilde\gamma) = k \cdot nr(\tilde r)$ for $\tilde r \in \fN(\R)$.
Again noting that no infinite prime of $K$ ramifies in $B$, the
Hasse-Schilling-Maass theorem (Theorem 33.15 of \cite{Reiner-book})
implies there exists a $b \in B^\times$ with $nr(b) = k$.  Thus
$nr(\tilde \gamma) = nr(b)\cdot nr(\tilde r)$, hence $nr(b^{-1})\cdot
nr(\tilde\gamma)\cdot nr(\tilde r^{-1})= (1) \in J_K$ which implies
$B^\times \tilde \gamma\ \fN(\R) = B^\times b^{-1} \tilde \gamma
\,\tilde r^{-1}\ \fN(\R) \in B^\times J_B^1 \fN(\R)$ as claimed.

To continue with injectivity, suppose that there exist $\tilde \alpha,
\tilde \beta \in J_B$ with $nr (B^\times \tilde \alpha\, \fN(\R)) = nr
(B^\times \tilde \beta\, \fN(\R))$. Then $K^\times nr(\tilde\alpha)
nr(\fN(\R)) = K^\times nr(\tilde\beta)\ nr(\fN(\R))$ which implies
$nr(\tilde\alpha^{-1}\tilde \beta) \in K^\times nr(\fN(\R))$. By the
claim, we have that $\tilde\alpha^{-1}\tilde \beta \in B^\times J_B^1
\fN(\R)$.  As above, it is easy to check that $B^\times J_B^1$ is a
normal subgroup of $J_B$, being the kernel of the induced homomorphism
$nr: J_B \to J_K/K^\times$, so that $\tilde \beta \in \tilde \alpha
B^\times J_B^1 \fN(\R) = B^\times J_B^1 \tilde\alpha \fN(\R)$. By
VI.iii and VII of \cite{Frohlich-lf}, we have that $J_B^1 \subset
B^\times \tilde \gamma \fN(\R) \tilde\gamma^{-1}$ for any $\tilde \gamma
\in J_B$, so choosing $\tilde \gamma = \tilde \alpha$, we have
\[
\tilde\beta \in  B^\times J_B^1 \tilde\alpha \fN(\R) \subseteq 
B^\times \tilde\alpha \fN(\R),
\]
so $B^\times \tilde \beta \fN(\R) \subseteq B^\times \tilde\alpha
\fN(\R)$, and by symmetry, we have equality.
\end{proof}

While it is well-known that the type number is finite (the type number
of an order is trivially bounded above by its class number and the class number
is finite (26.4 of \cite{Reiner-book})), we establish a stronger result
in our special case of central simple algebras of dimension $p^2$ over
$K$.  We show that the type number is a power of $p$; more
specifically, we show that

\begin{theorem}\label{thm:genus_p_group}Let $\R$ be a maximal order in
  a central simple algebra of dimension $p^2$ over a number field
  $K$.  Then the group $ K^\times \backslash J_K/nr(\fN(\R))$ is an
  elementary abelian group of exponent $p$.
\end{theorem}

\begin{proof}
Consider the quotient $J_K/nr(\fN(\R))$.  Each factor in the product
has the form $K_\nu^\times /nr(\N(R_\nu))$.  From above, we see that
this quotient is trivial when $\nu$ is infinite or finite and
ramified.  For finite split primes, $K_\nu^\times /nr(\N(R_\nu)) =
K_\nu^\times/(\O_v^\times (K_\nu^\times)^p)\cong \Z/p\Z$. So it
follows that $J_K/nr(\fN(\R))$ is an abelian group of exponent $p$.
The canonical homomorphism $J_K/nr(\fN(\R)) \to  K^\times \backslash
J_K/nr(\fN(\R))$ is surjective, so the resulting quotient is finite,
abelian, and of exponent $p$ which completes the proof.
\end{proof}

\subsection{The class field associated to a maximal order}
We have seen above that the distinct isomorphism classes of maximal
orders in $B$ (i.e., the isomorphism classes in the genus of any given
maximal order $\R$) are in one-to-one correspondence with the double
cosets in the group $G = K^\times \backslash J_K/nr(\fN(\R))$.  Put
$H_\R = K^\times nr(\fN(\R))$ and $G_\R = J_K/H_\R$.  Since $J_K$ is
abelian, $G$ and $G_\R$ are naturally isomorphic, and since $H_\R$
contains a neighborhood of the identity in $J_K$, it is an open
subgroup (Proposition II.6 of \cite{Higgins}).

Since $H_\R$ is an open subgroup of $J_K$ having finite index, there
is by class field theory \cite{Lang-ANT}, a class field $K(\R)$
associated to it.  The extension $K(\R)/K$ is an abelian extension
with $Gal(K(\R)/K) \cong G_\R = J_K/H_\R$ and with $H_\R = K^\times
N_{K(\R)/K}(J_{K(\R)})$.  Moreover, a prime $\nu$ of $K$ (possibly
infinite) is unramified in $K(\R)$ if and only if $\O_\nu^\times
\subset H_\R$, and splits completely if and only if $K_\nu^\times
\subset H_\R$.  Here if $\nu$ is archimedean, we take $\O_\nu^\times =
K_\nu^\times$.  From our computations at the beginning of this
section, we saw that $nr(\N(\R_\nu)) = K_\nu^\times$ or $\O_\nu^\times
(K_\nu^\times)^p$.  In particular $K(\R)/K$ is an everywhere unramified
extension of $K$.

\begin{prop}\label{prop:generators} Let $S$ be any finite set of primes
  in $K$ which includes the infinite primes.  The group $G_\R$ can be
  generated by cosets having representatives of the form $e_{\nu_i} =
  (1, \dots, 1, \pi_{\nu_i}, 1, \dots )$ for $\nu_i \notin S$,
  $\pi_{\nu_i}$ a uniformizer in $K_{\nu_i}$.
\end{prop}

\begin{proof}
  Artin reciprocity gives the exact sequence
\[\xymatrix{1\ar@{->}[r]& H_\R\ar@{->}[r]&  J_K\ar@{->}[r]^(.23)\Phi&
  J_K/H_\R \cong Gal(K(\R)/K) \ar@{->}[r]& 1},
 \]
with $\Phi$ the Artin map.  By the Chebotarev density theorem, there
are an infinite number of primes of $K$ in the preimage of each
element of $Gal(K(\R)/K)$ under $\Phi$.  The $e_{\nu_i}$ are the
images of those primes in $J_K$.
\end{proof}

We shall denote the generators of $G_\R\cong (\Z/p\Z)^m$ as
$\{\overline e_{\nu_i}\}_{i=1}^m$ where the $e_{\nu_i}$ are the ideles
of the previous proposition.  Let $L/K$ be a field extension of degree
$p$.  We now show that the generators $\{\overline e_{\nu_i}\}$ can be
chosen so that the $K$-primes $\nu_i$ have certain splitting
properties in $L$.  We shall use the symbol $(\fP, L/K)$ to denote the
Frobenius automorphism for an unramified prime $\fP$ of $L$ when $L/K$
is arbitrary, but also $(\nu,L/K)$ viewed as the Artin map for a prime
$\nu$ of $K$ when $L/K$ is an abelian extension.

\begin{prop}\label{prop:splitting-properties} With the notation as
  above, we have:
  \begin{compactenum}
    \item If $L \subset K(\R)$, then we may assume that $G_\R$ is
      generated by elements $\{\overline e_{\nu_i}\}$ where $\nu_i$
      splits completely in $L$ for $i > 1$, and $\nu_1$ is inert in
      $L$.
\item If $L \not\subset K(\R)$ then we may assume that $G_\R$ is
  generated by elements $\{\overline e_{\nu_i}\}$ where $\nu_i$ splits
  completely in $L$ for all $i\ge 1$.
  \end{compactenum}
\end{prop}

\begin{remark}
  Recall that $[K(\R):K] = p^m = t(\R)$ for $m\ge 0$.  The condition
  $L \subset K(\R)$ clearly forces $m \ge 1$, however when $L
  \not\subset K(\R)$, it is possible that the type number equals 1,
  though in that case the second part of the proposition is vacuously true.
\end{remark}

\begin{proof}
  First suppose that $L \subset K(\R)$.  Since $K(\R)/K$ is abelian,
  Galois theory provides the following exact sequence:
\[\xymatrix{1\ar@{->}[r]& Gal(K(\R)/L)\ar@{^{(}->}[r]^\iota&  Gal(K(\R)/K)\ar@{->}[r]^(.55){\textrm{res}_L}&
 Gal(L/K) \ar@{->}[r]& 1}.
 \]
 Let $\sigma \in Gal(K(\R)/L)$.  Viewing $Gal(K(\R)/L)\subseteq
 Gal(K(\R)/K)$, we can (by Chebotarev) write $\sigma = (\nu,K(\R)/K)$
 for an unramified prime $\nu$ of $K$.  From the exact sequence,
 $\sigma|_L = 1$, but $\sigma|_L = (\nu, L/K) = 1$ which implies $\nu$
 splits completely in $L$.  Now let $\tau$ be any element of
 $Gal(K(\R)/K)$ not in $Gal(K(\R)/L)$.  Writing $\tau = (\mu,K(\R)/K)$
 ($\mu$ unramified), we see that since $\tau|_L \ne 1$, we have $(\mu, L/K)
 \ne 1$ which means $\mu$ does not split completely in $L$.  But $L/K$
 having prime degree means $\mu$ is inert in $L$.  Note that for any
 $\tau \notin Gal(K(\R)/L)$, $Gal(K(\R)/K)$ is the internal direct
 product of $\la \tau\ra$ and $Gal(K(\R)/L)$ from which the assertion
 follows.  In particular, if $\mu$ is any prime of $K$ inert in $L$,
 $Gal(K(\R)/K)$ is generated by $(\mu, K(\R)/K)$ and $Gal(K(\R)/L)$.

 Next we assume that $L \not\subset K(\R)$; there are two cases
 corresponding to whether $L/K$ is Galois or not.  We begin with the
 case that $L/K$ is Galois.  Since $L/K$ has prime degree, $L
 \not\subset K(\R)$ implies that $L\cap K(\R) = K$, hence the
 composite extension $K(\R)L/L$ is abelian with $Gal(K(\R)L/L) \cong
 Gal(K(\R)/K)$ via restriction.  Let $\sigma$ be any nontrivial
 element of $Gal(K(\R)L/L)$, and write $\sigma = (\fP,K(\R)L/L)$ as an
 Artin symbol by Chebotarev, where $\fP$ is a prime of $L$ unramified
 over $K$.  Put $\nu = \fP \cap K$.  We claim that we may also assume
 that the inertia degree $f(\fP/\nu) = 1$. To see this note that the
 set of primes of $L$ having inertia degree (over $\Q$) greater than
 one has density 0, and Chebotarev guarantees we may
 choose $\fP$ from a set of primes of positive density, hence the
 claim.  Thus every nontrivial element $\sigma$ of $Gal(K(\R)L/L)$ has
 the form $\sigma = (\fP,K(\R)L/L)$ with ramification index
 $e(\fP/\nu) = 1$ and inertia degree $f(\fP/\nu) =1$, where $\nu = \fP
 \cap K$.  Since $(\fP,K(\R)L/L)|_{K(\R)} = (\nu,
 K(\R)/K)^{f(\fP/\nu)} = (\nu, K(R)/K)$, every nontrivial element of
 $Gal(K(\R)/K)$ has the form $(\nu, K(\R)/K)$ where $\nu$ is a prime of
 $K$ which splits completely in $L$ as desired.

 Finally, we assume $L\not\subset K(\R)$ and $L/K$ is not Galois.  Let
 $\hL$ denote the Galois closure of $L/K$.  It is well-known (see p58
 of \cite{Neukirch-book}), that a prime $\nu$ of $K$ splits completely
 in $L$ if and only if it splits completely in $\hL$. So if we can
 show that $\hL \cap K(\R) = K$, the result in this case will follow
 from the previous one.  To that end, let $F = \hL \cap K(\R)$.  Then
 $K \subset F \subset K(\R)$, so $[F:K]$ is a power of the prime $p$.
Now $[L:K] = p$ implies $[\hL:K] \mid p!$, and since
 $F\subset \hL$, we have $[F:K] \mid p!$. So $[F:K] = 1$ or
 $p$. Suppose $[F:K] = p$.  As $K \subset F \subset K(\R)$, $F/K$ is
 an abelian extension of $K$, so in particular $F \ne L$, which
 implies $F\cap L = K$.  Thus $FL/L$ is Galois with $Gal(FL/L) \cong
 Gal(F/K)$.  In particular, $[FL:K] = p^2$.  But $FL \subset \hL$, so
 $p^2 = [FL:K] \mid [\hL:K] \mid p!$, a contradiction.  Thus $\hL \cap
 K(\R) = K$ as desired.
\end{proof}

\subsection{Parametrizing the isomorphism classes}
Let $\R$ be a fixed maximal order in $B$, and recall $G_\R =
J_K/K^\times nr(\fN(\R)) \cong (\Z/p\Z)^m$, $m\ge 0$.  Let
$\{e_{\nu_i}\}_{i=1}^m \subset J_K$ so that their images $\{\overline
e_{\nu_i}\}_{i=1}^m$ generate $G_\R$.  By
Proposition~\ref{prop:generators}, we may choose the $\nu_i$ to avoid
any finite set of primes; for now we simply assume that all the
$\nu_i$ are non-archimedean and split in $B$, in particular that
$B_{\nu_i} \cong M_p(K_{\nu_i})$.  For each $\nu_i$ we shall regard
$\R_{\nu_i}$ as a vertex in the building for $SL_p(K_{\nu_i})$, and
let $C_i$ be any chamber containing $\R_{\nu_i}$.  We may assume that
in a given labeling of the building, $\R_{\nu_i}$ has type zero
\cite{Ronan}, and we label the remaining vertices of the chamber $C_i$
as $\R_{\nu_i}^{(k)}$, (having type $k$) $k = 1, \dots, p-1$, putting
$\R_{\nu_i}^{(0)} = \R_{\nu_i}.$

Given a $\gamma = (\gamma_i) \in (\Z/p\Z)^m$, we define $p^m$ distinct
maximal orders, $D^\gamma$, in $B$ via the local-global correspondence
by providing
the following local data:

\begin{equation}\label{eqn:parametrization}
\D_\nu^\gamma =
\begin{cases}
  \R_{\nu_i}^{(\gamma_i)}&\textrm{if $\nu = \nu_i$}\\
  \R_\nu&\textrm{otherwise}.
\end{cases}
\end{equation}

We claim that any such collection of maximal orders parametrizes the
genus of $\R$, that is given any maximal order $\E$, there is a unique $\gamma
\in  (\Z/p\Z)^m$, so that $\E \cong \D^\gamma$.  To show this, let
$\fM$ denote the set of all maximal orders in $B$, and define a map
$\rho: \fM \times \fM \to G_\R$ as follows.

Let $\R_1, \R_2 \in \fM$.  For $\nu$ a finite prime of $K$ (split in
$B$), we have defined the type distance between their localizations:
$td_\nu(\R_{1\nu}, \R_{2\nu}) \in \Z/p\Z$. For $\nu$ archimedean or
$\nu$ finite and ramified in $B$, define $td_\nu(R_{1\nu},
R_{2\nu})=0$.  Recall that since $\R_{1\nu} = \R_{2\nu}$ for almost
all $\nu$, $td_\nu(\R_{1\nu}, \R_{2\nu})=0$ for almost all primes
$\nu$. Let $\rho(\R_1, \R_2)$ be the image in $G_\R$ of the idele
$(\pi_\nu^{td_\nu(\R_{1\nu}, \R_{2\nu})})$.  Note that while the idele
is not well-defined, its image in $G_\R$ is since the local factor at
the finite split primes has the form $K_\nu^\times/\O_\nu^\times
(K_\nu^\times)^p$.

We now show that any such collection of maximal orders given as the
$\D^\gamma$ parametrizes the genus.

\begin{prop}\label{prop:parametrize}
Let $\R$ be a fixed maximal order in $B$, and consider the collection
of maximal orders $\D^\gamma$ defined above.
\begin{compactenum}
\item If $\E$ is a maximal order in $B$ and $\E \cong \R$, then
  $\rho(\R, \E)$ is trivial.
\item If $\E \cong \E'$ are maximal orders in $B$, then $\rho(\R, \E) = \rho(\R,\E')$.
    \item $\D^\gamma \cong \D^{\gamma'}$ if and only if $\gamma = \gamma'$.
\end{compactenum}
\end{prop}

\begin{proof}
  For the first assertion, we may assume that $\E = b\R b^{-1}$ for
  some $b \in B^\times$ by Skolem-Noether, which of course means
  $\E_\nu = b \R_\nu b^{-1}$ for each prime $\nu$.  For a finite prime
  which splits in $B$, we may take $\R_\nu = \End(\Lambda_\nu)$ for
  some $\O_\nu$-lattice $\Lambda_\nu$, and so $\E_\nu
  = \End(b\Lambda_\nu)$.  It follows that
\[td_\nu(R_\nu,\E_\nu) \equiv \ord_\nu(\det(b^{-1})) \equiv
\ord_\nu(nr(b^{-1})) \pmod p,
\]
and since $G_\R$ is trivial at the archimedean primes and the finite
primes which ramify in $B$, we conclude that
$\rho(\R, \E) = (\overline{nr(b^{-1})}) = (\bar 1)$ in $G_\R = J_K/K^\times
nr(\fN(\R))$ as $(nr(b^{-1}))$ is in the image of $K^\times$ in $J_K$.

To see the second assertion, we have as above $\E' = b\E b^{-1}$ for
some $b \in B^\times$ and so $\E'_\nu = b \E_\nu b^{-1}$ for each
prime $\nu$.  If we write $\R_\nu = \End(\Lambda_\nu)$ and $\E_\nu
= \End(\Gamma_\nu)$ for
   $\O_\nu$-lattices $\Lambda_\nu$ and $\Gamma_\nu$, then $\E'_\nu
   = \End(b\Gamma_\nu)$.  Considering the elementary divisors of the
   lattices $\Lambda_\nu$, $\Gamma_\nu$ and $b\Gamma_\nu$, we easily
   see that $td_\nu(\R_\nu, \E'_\nu) \equiv td_\nu(R_\nu,\E_\nu) +
   \ord_\nu(\det(b^{-1})) \pmod p$, from which the result follows as in the
   first case.

   For the last assertion, we need only show one direction.  Fix a
   prime $\nu=\nu_i$ among the finite number used to determine the
   parametrization $\D^\gamma$.  Then we are comparing
   $R_{\nu}^{(\gamma_i)}$ and $R_{\nu}^{(\gamma_i')}$. Since the
   $\R_\nu^{(k)}$ $k = 0, \dots, p-1$ are the vertices of a fixed
   chamber in the affine building for $SL_p(K_\nu)$, they can be
   realized (p 362 of \cite{Abramenko-Brown}) as $R_\nu^{(k)}
   = \End_{\O_\nu}(\Lambda^{(k)})$ with $\Lambda^{(k)} = \O_\nu \pi
   \omega_1 \oplus \cdots \oplus \O_\nu\pi\omega_k \oplus \O_\nu
   \omega_{k+1} \oplus \cdots \oplus \O_\nu\omega_p$.  Here the set
   $\{\omega_i\}$ a basis of a vector space $V/K_\nu$ through which we
   have identified $B_\nu = \End_{K_\nu}(V)$.  It follows that
   $td_\nu(R_{\nu}^{(\gamma_i)},R_{\nu}^{(\gamma_i')}) = \gamma_i' -
   \gamma_i \pmod p$.  It is now easy to see that if $\gamma \ne
   \gamma'$, then $\rho(\D^\gamma, \D^{\gamma'}) \ne (\bar 1)$, so $\D^\gamma
   \not\cong \D^{\gamma'}$.
\end{proof}

\subsection{Selective Orders and the Main
  Theorem}\label{sec:selective_orders}

We reestablish the notation from the introduction.  Let $p$ an odd
prime, $B$ a central simple algebra of dimension $p^2$ over a number
field $K$, and $L/K$ a field extension of degree $p$ which satisfies
$L \subset B$.  Let $\O_K$ denote the ring of integers of $K$, and let
$\Omega$ denote a commutative $\O_K$-order of rank $p$ in $L$.
Necessarily $\Omega$ is an integral domain with field of fractions
equal to $L$, and we have seen that $\Omega$ is contained in a maximal order
$\R$ of $B$ which we now fix.

Given that $\Omega$ is contained in $\R$, the question is into which
other isomorphism classes in the genus of $\R$ does $\Omega$ embed?
Recall that since $\R$ is maximal, this simply asks into which
isomorphism classes of maximal orders in $B$ does $\Omega$ embed?
The general case is that it embeds in all the isomorphism classes,
but when it does not, we follow \cite{Chinburg-Friedman} and call
$\Omega$ selective.  Selectivity is characterized by our main theorem.

\begin{theorem}\label{thm:main}
  With the notation fixed as above, every maximal order in $B$
  contains a conjugate (by $B^\times$) of $\Omega$ except when the
  following conditions hold:
  \begin{compactenum}
  \item $L \subseteq K(\R)$, that is $L$ is contained in the class
    field associated to $\R$.
 \item Every prime ideal $\nu$ of
   $K$ which divides $N_{L/K}(\f)$ splits in $L/K$.
  \end{compactenum}
  Suppose now that both conditions (1) and (2) hold.  Then precisely
  one-$p$th of the isomorphism classes of maximal orders contain a
  conjugate of $\Omega$.  Those classes are characterized by means of
  the Frobenius $\Frob_{L/K}$ as follows: $\E$ is a maximal order
  which contains a conjugate of $\Omega$ if and only if
  $\Frob_{L/K}(\rho(\R, \E))$ is trivial in $Gal(L/K)$.
\end{theorem}

First we give some examples of selective and non-selective orders. Let
$S_\infty$ denote the set of infinite primes of $K$ and let $Ram(B)$
denote the set of primes in $K$ which ramify in $B$.

\begin{example}
  Let $p$ be an odd prime, $K$ a number field with class number $p$,
  let $B = M_p(K)$, and $\R = M_p(\O_K)$.  Then $G_\R =
  J_K/K^\times (J_K \cap (\prod_{\nu\not\in
    S_\infty}(K_\nu^\times)^p\O_\nu^\times\prod_{\nu\in
    S_\infty}K_\nu^\times))\cong C_K/C_K^p\cong C_K$, where $C_K$ is
  the ideal class group of $K$, and $C_K^p$ the subgroup of $p$th
  powers.  We conclude the type number $t(\R) = |G_\R| = p$.  This
  means that $[K(\R):K] = p$ and $K(\R)/K$ is an everywhere unramified
  abelian extension of $K$, so $K(\R) \subset \widetilde K$, where
  $\widetilde K$ is the Hilbert class field of $K$.  Degree
  considerations force $K(\R) = \widetilde K$.  Put $L = K(\R) =
  \widetilde K$.  Because $B$ is everywhere split, $L$ embeds into $B$.
  So we have $L \subseteq K(\R)$, $L \subset B$.  This means that
  $\O_L$ is selective as established in \cite{Arenas-Carmona},
  \cite{Chevalley-book} as well as our main theorem.  Now let $\nu$ be
  a prime of $K$, necessarily unramified in $L = \widetilde K$, and
  consider the order $\Omega = \O_K + \nu \O_L$.  We easily see that
  $\nu\O_L \subset \f$ which implies $\f \mid \nu \O_L$, hence
  $N_{L/K}(\f) \mid N_{L/K}(\nu \O_L) = \nu^p \O_K$ whether $\nu$ is
  inert or splits completely in $L$.  Since $\f \ne \O_L$, we see that
  $\nu \mid N_{L/K}(\f)$, so by condition (2) of the theorem, in the
  case that $\nu$ is inert, we see $\Omega$ is not selective, but when
  $\nu$ splits completely, $\Omega$ is selective.
\end{example}

Indeed, given the theorem, we have the following interesting
corollary.

\begin{cor}
  Suppose there exists a field extension $L/K$ with $[L:K]=p$ which
  embeds into $B$, and which contains an order $\Omega \subseteq \O_L$
  which is selective.  Then $B\cong M_p(K)$.  Said alternatively,
  suppose we are given any number field $L/K$ of degree $p$, and any
  suborder $\Omega \subset \O_L$.  If $B$ is a degree $p$ division
  algebra, then $\Omega$ embeds into every maximal order in $B$ if and
  only if $L$ embeds into $B$. In particular, a degree $p$ division
  algebra admits no selective orders.
\end{cor}

\begin{proof}
  Given $L \subset B$ and $\Omega$ selective, we must have $L
  \subseteq K(\R)$.  Now $K(\R)$ is the class field associated to the
  subgroup $\ds H_\R = K^\times(J_K\cap [\prod_{\nu \in S_\infty\cup Ram(B)}
  K_\nu^\times \times \prod_{\nu\not\in S_\infty\cup Ram(B)}
  \O_\nu^\times (K_\nu^\times)^p])$.  In particular, if $\nu \in
  Ram(B)$, then $K_\nu^\times \subset H_\R$ which means that $\nu$
  splits completely in the class field $K(\R)$, hence in $L$.  But
  this violates the Albert-Brauer-Hasse-Noether theorem which implies
  that no prime that ramifies in $B$ splits in $L$.
\end{proof}

We give the proof of the main theorem via a sequence of propositions.

\begin{prop}\label{prop:Omega-1} Let $\Omega$ denote an
  $\O_K$-order which is an integral domain whose field of fractions
  $L$ is a degree $p$ extension of $K$ which is contained in $B$.  We
  assume that $\Omega$ is contained in a fixed maximal order $\R$ of
  $B$.  If $L \not\subset K(\R)$ then every isomorphism class of
  maximal order in $B$ contains a conjugate (by $B^\times$) of $\Omega$.
\end{prop}

\begin{proof}
  Note that if the type number $t(\R) = 1$, the proposition is
  obviously true, so we assume $t(\R) = [K(\R):K] = p^m$ with $m \ge
  1$.  By Proposition~\ref{prop:splitting-properties}, we may choose
  elements $\{e_{\nu_1}, \dots, e_{\nu_m}\}\subset J_K$ so that the
  cosets $\{\overline e_{\nu_i}\}$ generate $G_\R = J_K/K^\times
  nr(\fN(\R))$, and so that the primes $\nu_i$ of $K$ are finite and
  split completely in $L$.  Since $[L:K] = p$, $L$ is a strictly
  maximal subfield of $B$ (section 13.1 of \cite{Pierce-book}) and
  consequently (Corollary 13.3 \cite{Pierce-book}), $L$ is a
  splitting field for $B$.  We claim that all the $\nu_i$ are split in
  $B$.  Fix $\nu = \nu_i$ and let $\fP$ be any prime of $L$ lying
  above $\nu$.  As $\nu$ splits completely in $L$, $[L_\fP :K_{\nu}] =
  1$. By Theorem 32.15 of \cite {Reiner-book}, $m_{\nu}$ which is the
  local index of $B_{\nu}/K_{\nu}$ must divide $[L_\fP :K_{\nu}]$,
  thus $B_{\nu} \cong M_p(K_{\nu})$, as desired.  Now, $L \subset
  B$ implies that $L\otimes_K K_{\nu} \cong \oplus_{\fP \mid \nu}
  L_\fP \cong K_{\nu}^p \hookrightarrow B\otimes_K K_{\nu} =
  B_{\nu}$. By a slight generalization of Skolem-Noether to
  commutative semisimple subalgebras of matrix algebras (Lemma 2.2 of
  \cite{Brzezinski-TwoThms}), we may assume we have a
  $K_{\nu}$-algebra isomorphism $\varphi: B_{\nu} \to M_p(K_{\nu})$
  such that \newline $\varphi(L)
  \subset
  \begin{pmatrix}
    K_{\nu}&&&0&\\&K_{\nu}&&\\&&\ddots&\\0&&&K_{\nu}\\
  \end{pmatrix}$ and hence $\varphi(\Omega) \subset \varphi(\O_L)
  \subset
  \begin{pmatrix}
    \O_{\nu}&&&0&\\&\O_{\nu}&&\\&&\ddots&\\0&&&\O_{\nu}\\
  \end{pmatrix}$.  By Corollary 2.3 of \cite{Shemanske-Split} all
  maximal orders containing $\diag(\O_{\nu}, \dots, \O_{\nu})$ have a
  prescribed form and lie in a fixed apartment in the affine building
  for $SL_p(K_\nu)$ and so it follows that by a rescaling of basis we
  may assume in addition that $\varphi(\R_\nu) = M_p(\O_\nu)$.

  With $\pi$ a uniformizer in $K_\nu$, let $\delta_k =
  \diag(\underbrace{\pi,\dots,\pi}_k,1,\dots,1)\in M_p(K_\nu)$, $k =
  0, \dots, p-1$, and define maximal orders $\E_k = \delta_k
  M_p(\O_\nu) \delta_k^{-1}$.  These are all maximal orders containing
  $\diag(\O_{\nu}, \dots, \O_{\nu})$, and are all the vertices of a
  fixed chamber in the building for $SL_p(K_\nu)$.  If we put
  $R_\nu^{(k)} = \varphi^{-1}(\E_k)$ for $k = 0, \dots p-1$, and $\nu
  \in \{\nu_1, \dots, \nu_m\}$ we then obtain a parametrization
  $\D^\gamma$ of the isomorphism classes of all maximal orders in $B$ as in
  Equation~(\ref{eqn:parametrization}).  Since $\Omega \subset \R$,
  and by construction $\Omega \subset  R_{\nu_i}^{(0)}
  \cap \cdots \cap R_{\nu_i}^{(p-1)}$ for each $\nu_i$, we have that
  $\Omega \subset \D_\nu^\gamma$ for all primes $\nu$ and all $\gamma$
  which is to say every isomorphism class of maximal order in $B$
  contains a conjugate of $\Omega$.
\end{proof}

Next we assume that condition (1) of the theorem holds, but
not condition (2). Note that since $L \subset K(\R)$ and $K(\R)/K$ is
  an everywhere unramified abelian extension, so is $L/K$.  Moreover, since
  $L/K$ is of prime degree (and Galois), any unramified prime splits completely
  or is inert.

\begin{prop}\label{Omega-2}
  Assume that $\Omega$ is an integral domain contained in $\R$ whose
  field of fractions $L \subset K(\R)$.  Assume that there is a prime
  $\nu$ of $K$ which divides $N_{L/K}(\f)$, the norm of the
  conductor $\f$ of $\Omega$, but which does not split completely in
  $L$.  Then every isomorphism class of maximal order in $B$ contains a
  conjugate of $\Omega$.
\end{prop}

\begin{proof}Since condition (2) is assumed not to hold, we may assume
  by the comments above that there is a prime $\nu$ of $K$ which
  divides $N_{L/K}(\f)$ and which is inert in $L$. Thus we may assume that
  $\nu\O_L \mid \f$. Our first goal is to show that $\Omega \subset
  \O_K + \nu \O_L$.

  We first assume that $\Omega$ has the form $\Omega = \O_K[a]$ for
  some $a \in \O_L$, and let $f$ be the minimal polynomial of $a$ over
  $K$.  Since $\Omega \otimes_{\O_K} K \cong  L$, $f$ is irreducible of degree
  $p$, and since $a$ is integral, $f \in \O_K[x]$.  By Proposition
  4.12 of \cite{Narkiewicz-book}, $\f = f'(a)\partial_{L/K}^{-1} =
  f'(a)\O_L$ since $L/K$ everywhere unramified implies that the
  different $\partial_{L/K} = \O_L$. So it follows that $f'(a) \equiv 0
  \pmod \nu.$ Put $\overline a = a + \nu \O_L$ and consider the tower
  of fields:
\[
O_K/\nu\O_K \subseteq \O_K/\nu\O_K[\overline a] \subseteq \O_L/\nu\O_L.
\]
 From top to bottom, this is a degree $p$ extension of finite fields
 since $\nu$ is inert in $L$, and the ring in the middle is a field
 since it is a finite integral domain.  Since the total extension has
 prime degree, there are two cases.

If $\O_K/\nu\O_K[\overline a] = \O_L/\nu\O_L$, then $\overline f$ (the
reduction of $f$ mod $\nu\O_K$) is irreducible and hence is the minimal
polynomial of $\overline a$. In particular $\overline f$ must be
separable polynomial since finite fields are perfect.  On the other
hand, $\overline f$ and $\overline{f'}$ share the common root
$\overline a$, so $\overline f$ is not separable, a contradiction.

Thus $\O_K/\nu\O_K[\overline a] = \O_K/\nu\O_K$ where we view
$\O_K/\nu\O_K$ embedded as usual in $\O_L/\nu\O_L$.  Thus $\overline a
= a + \nu \O_L \in \O_K/\nu\O_K$ which means that $a + \nu \O_L = b + \nu
\O_L$ for some $b \in \O_K$.  This means that $a \in b + \nu\O_L$
which in turn means that $\Omega = \O_k[a] \subset \O_K + \nu\O_L$.

Now consider the general case of an order $\Omega$.  We show $\Omega
\subset \O_K + \nu\O_L$ by showing each element of $\Omega$ is in
$\O_K + \nu\O_L$.  Choose $a \in \Omega$.  Without loss assume $a
\notin \O_K$.  Then $\O_K[a]$ is an integral domain whose field of
fractions is all of $L$ since $L/K$ has prime degree.  Moreover, $\f
\mid \mathfrak f_{\O_K[a]/\O_K}$, so we may use the same inert prime
$\nu$ for all elements of $\Omega$, and the special case now implies
the general result.

By Proposition~\ref{prop:splitting-properties} (and its proof), we may
choose primes $\nu_1, \dots, \nu_m$ of $K$ so that the $\{\overline
e_{\nu_i}\}$ generate $G_\R$, where $\nu_i$ splits completely in $L$
for $i > 1$ and where $\nu_1$ is inert in $L$.  Consider the situation
locally at $\nu=\nu_1$.  We have that $\Omega_\nu \subset \O_{\nu} +
\nu\O_{L_\nu} \subset \O_{L_\nu}$. As in the previous proposition, we
have a $K_\nu$-algebra isomorphism $\varphi: B_\nu \to
M_p(K_\nu)$. Let $\D_\nu$ be a maximal order in $M_p(K_\nu)$
containing $\varphi(\O_{L_\nu})$ and hence $\varphi(\Omega)$.  Since
all maximal orders in $M_p(K_\nu)$ are conjugate, writing $\D_\nu
= \End_{\O_{\nu}}(\Lambda_\nu)$ for some $\O_\nu$-lattice
$\Lambda_\nu$, we may assume that $\varphi$ is defined so that $\D_\nu
= M_p(\O_\nu)$.  As in the previous proposition, let $\delta_k =
\diag(\underbrace{\pi,\dots,\pi}_k,1,\dots,1)\in M_p(K_\nu)$, $k = 0,
\dots, p-1$, and define maximal orders $\D_\nu^{(k)} = \delta_k
M_p(\O_\nu) \delta_k^{-1}$.  One trivially checks that $\nu \D_\nu \in
\D_\nu^{(k)}$ for $k = 0, \dots, p-1$, so that $\varphi(\Omega)
\subset \varphi (\O_\nu + \nu \O_{L_\nu}) \subset \O_\nu + \nu \D_\nu
\subset \D_\nu^{(k)}$ for $k=0, \dots , p-1$.  Putting $\R_\nu^{(k)} =
\varphi^{-1}(\D_\nu^{(k)})$, we have $\Omega \subset \R_\nu^{(k)}$ for
$k= 0, \dots, p-1$, and we may use these $\R_\nu^{(k)}$ as part of the
parametrization of the isomorphism classes of maximal orders.  The
other primes $\nu_2, \dots, \nu_m$ all split completely in $L$, and
the previous proposition shows that $\Omega$ is contained in all the
local factors of our parametrization.  So as before, $\Omega$ is
contained in every isomorphism class of maximal order in $B$.
\end{proof}

Finally, we assume that conditions (1) and (2) hold, and show that
$\Omega$ is contained in only one-$p$th of the isomorphism classes of $\R$.
We require a small technical lemma.

\begin{lemma} As above, let $\Omega$ denote an $\O_K$-order which is an
  integral domain whose field of fractions $L$ is a cyclic extension
  of $K$ having prime degree $p$.  We assume that $L$ is contained in
  $B$, and let $\nu$ be a prime of $K$ which is inert in $L$.  If $\nu
  \nmid N_{L/K}(\f$), then there exists an $a \in \Omega\setminus \O_K$ so that
  $\nu \nmid N_{L/K}(\ff_{\O_K[a]/\O_K})$.
\end{lemma}

We remark that this lemma represents a statement that in this narrow
context $\Omega$ has no common non-essential discriminantal divisors,
see \cite{Narkiewicz-book}, a frequent obstruction to assuming the an
order has a power basis.

\begin{proof} First note that since $\nu$ is inert in $L$, the stated
  condition on the conductor $\f$ is equivalent to $\nu\O_L \nmid \f$.
  Let $a \in \Omega$, and consider the tower of fields (the quotient
  ring in the middle being a finite integral domain):
\[\xymatrix{\O_K/\nu\O_K\ar@{^{(}->}[r]&  (\O_K/\nu\O_K)[a + \nu\O_L] \ar@{^{(}->}[r]&
 \O_L/\nu\O_L}.
 \]
Since $\nu$ is inert in $L$, $[\O_L/\nu\O_L : \O_K/\nu\O_K] = p$, so the
field in the middle coincides with one of the ends.  If $(\O_K/\nu\O_K)[a
+ \nu\O_L] = \O_K/\nu\O_K$, then $a +\nu\O_L = b + \nu \O_L$ for some $b
\in \O_K$, hence $\O_K[a] \subset \O_K + \nu\O_L$.  If this happens
for each $a \in \Omega$, then $\Omega \subset \O_K + \nu\O_L$.
Consider the conductors of these orders: Certainly, $\ff_{(\O_K +
  \nu\O_L)/\O_K} \mid \f$, and $\nu\O_L \subset \ff_{(\O_K +
  \nu\O_L)/\O_K} = \{ x \in \O_L \mid x\O_L \subseteq \O_K +
  \nu\O_L\}.$  But as $\nu$ is inert in $L$, $\nu\O_L$ is a maximal
  ideal, and since $\O_K + \nu\O_L \ne \O_L$, $\ff_{(\O_K +
  \nu\O_L)/\O_K} \ne \O_L$, so $\ff_{(\O_K +
  \nu\O_L)/\O_K} = \nu\O_L$.  This implies $\nu\O_L \mid \f$, a
contradiction.

So there must exist an $a \in
\Omega\setminus \O_K$ so that $a \not\in \O_K + \nu\O_L$. This implies
$\O_K[a]/(\nu\O_L \cap \O_K[a]) \not\cong (\O_K/\nu\O_K)$, so we have
$\O_L/\nu\O_L \cong \O_K[a]/(\nu\O_L \cap \O_K[a])$.  By Proposition
4.7 of \cite{Narkiewicz-book}, $\nu \nmid
N_{L/K}(\ff_{\O_K[a]/\O_K})$, as required.
\end{proof}

\begin{prop} Suppose now that conditions (1) and (2) hold.  Then
  precisely one-$p$th of the isomorphism classes of maximal orders in $B$
  contain a conjugate of $\Omega$.  Those classes are characterized by
  means of the Frobenius $\Frob_{L/K}$ as follows: $\E$ is a maximal
  order which contains a conjugate if and only if
  $\Frob_{L/K}(\rho(\R, \E))$ is trivial in $Gal(L/K)$.
\end{prop}

\begin{remark}
  First we indicate our meaning of $\Frob_{L/K}(\rho(\R, \E))$.
  Recall that $\rho(\R, \E) \in G_\R = J_K/H_\R$, and by Artin
  reciprocity, there is an isomorphism $G_\R \to Gal(K(\R)/K)$ given
  by the Artin map which we denote here as $\Frob_{K(\R)/K}$.  Thus
  $\Frob_{k(\R)/K}(\rho(\R, \E))$ is an element of $Gal(K(\R)/K)$
  which we restrict to $L$.  The Artin map is also compatible with
  restriction giving that $\Frob_{k(\R)/K}(\rho(\R, \E))|_L =
  \Frob_{L/K}(\rho(\R, \E))$.
\end{remark}

\begin{proof}
  We have assumed that $\Omega \subset \R$, and suppose that $\E$ is
  another maximal order in $B$.  We shall show that $\E$ contains a
  conjugate of $\Omega$ if and only if $Frob_{L/K}(\rho(\R, \E))$ is
  trivial in $Gal(L/K)$.  We first show that $Frob_{L/K}(\rho(\R,
  \E))$ non-trivial in $Gal(L/K)$ implies that $\E$ does not contain a
  conjugate of $\Omega$.  We proceed by contradiction and assume that
  $\E$ does contain a conjugate of $\Omega$.  Then there is $b \in
  B^\times$ so that $\Omega \subset \E^* = b\E b^{-1}$.  By
  Proposition~\ref{prop:parametrize}, $Frob_{L/K}(\rho(\R,\E)) =
  Frob_{L/K}(\rho(\R, \E^*)) \ne (\bar 1)$, so there exists a prime $\nu$ of
  $K$ which is inert in $L$ so that $td_\nu(\R_\nu, \E^*_\nu)
  \not\equiv 0 \pmod p$.  In particular, $\R_\nu$ and $\E^*_\nu$ are
  of different types, so indeed $\R_\nu \ne \E_\nu^*$.  View these two
  maximal orders as vertices in the building for $SL_p(K_\nu)$, and
  choose an apartment containing them.  We may assume that a basis
  $\{\omega_i\}$ for the apartment is chosen in such a way that one
  maximal order is $\End_{O_\nu}(\oplus \O_\nu \omega_i)$ which we
  identify with $M_p(\O_\nu)$ and the other with $\End_{O_\nu}(\oplus
  \O_\nu \pi^{m_i}\omega_i)$ ($0\le m_1\le \cdots \le m_p$), which is
  identified with $\diag(\pi^{m_1}, \dots,
  \pi^{m_p})M_p(\O_\nu)\diag(\pi^{m_1}, \dots, \pi^{m_p})^{-1} =
  \Lambda(m_1, \dots, m_p) = \begin{pmatrix}
    \O&\nu^{m_1 - m_2}&\nu^{m_1-m_3}& \dots&\nu^{m_1-m_p}\\
    \nu^{m_2-m_1}&\O&\nu^{m_2-m_3}&\dots&\nu^{m_2 - m_p}\\
    \nu^{m_3-m_1}&\nu^{m_3-m_2}&\ddots&\dots&\nu^{m_3 - m_p}\\
    \vdots&\vdots&&\O&\vdots\\
    \nu^{m_p - m_1}&\dots&&\nu^{m_{p} - m_{p-1}}&\O
  \end{pmatrix}
  $.  We may assume without loss that $m_1 = 0$ since $\End(L)$ is
  unchanged by the homothety class of the lattice $L$, and since
  $\R_\nu \ne \E^*_\nu$, we must have that $m_p \ge 1$.  Let $\ell$ be
  the smallest index so that $m_\ell \ge 1$.  Note that the image of 
  $M_p(\O_\nu) \cap \Lambda(m_1, \dots, m_p)$ under the projection
  from $M_p(\O_\nu) \to M_p(\O_\nu/\nu\O_\nu)$ is contained in
  $\begin{pmatrix}
    M_{\ell -1}(\O_\nu/\nu\O_\nu)&*\\ 0&M_{p-\ell+1}(\O_\nu/\nu\O_\nu)
  \end{pmatrix}$.

  By the lemma, we can choose an element $a \in \Omega\setminus \O_K$,
  with $\nu \nmid N_{L/K}(\mathfrak f_{\O_K[a]/\O_K})$, and since
  $L/K$ has prime degree, $L$ is the field of fractions of $\O_K[a]$.
  This allows us to invoke the Dedekind-Kummer theorem (Theorem 4.12
  of \cite{Narkiewicz-book}).  Let $f$ be the minimal polynomial of
  $a$ over $K$.  Because $L/K$ has prime degree and $a$ is integral,
  $f \in \O_K[x]$ and is irreducible.  Since Dedekind-Kummer applies,
  we consider the factorization of $\overline f \in (\O_K/\nu\O_K)[x]$
  which will mirror the factorization of $\nu$ in the field $L$.  Of
  course we know that $\nu$ is inert, so that $\overline f$ is
  irreducible in $(\O_K/\nu\O_K)[x]$.  Now since $L \subset B$, we can
  view $a \in B_\nu \cong M_p(K_\nu)$.  Without loss we identify
  $B_\nu$ with the matrix algebra.  Let $F$ be the characteristic
  polynomial of $a$ over $K_\nu$ which, because $a$ is integral, will
  have coefficients in $\O_\nu[x]$. Consider $\overline F \in
  (\O_\nu/\nu\O_\nu)[x] \cong (\O_K/\nu\O_K)[x]$.  Now both $\overline
  f$ and $\overline F$ are polynomials of degree $p$ in
  $(\O_K/\nu)[x]$ having $a$ for a root.  We know that $\overline f$
  is irreducible, so $\overline f \mid \overline F$, from which it
  follows that $\overline F = \overline f$ by degree considerations,
  and hence is irreducible. On the other hand, since $a \in R_\nu \cap
  \E^*_\nu$ we know that its image under the projection from
  $M_p(\O_\nu) \to M_p(\O_\nu/\nu\O_\nu)$ lies in
  $\begin{pmatrix} M_{\ell -1}(\O_\nu/\nu\O_\nu)&*\\
    0&M_{p-\ell+1}(\O_\nu/\nu\O_\nu)
  \end{pmatrix}$ which means that $\overline F$ (the characteristic
  polynomial of $a$) will be reducible over $\O_\nu/\nu\O_\nu$ by the
  inherent block structure, a contradiction.

  Now we show the converse: Recall that $\Omega \subset \R$, and let
  $\E$ be another maximal order in $B$.  We shall show that if
  $Frob_{L/K}(\rho(\R, \E))$ is trivial in $Gal(L/K)$ then $\E$
  contains a conjugate of $\Omega$. By
  Proposition~\ref{prop:splitting-properties}, we may choose primes
  $\nu_1, \dots, \nu_m$ of $K$ so that the $\{\overline e_{\nu_i}\}$
  generate $G_\R$, where $\nu_i$ splits completely in $L$ for $i > 1$
  and where $\nu_1$ is inert in $L$.  Parametrize the isomorphism
  classes of maximal orders as
  in Equation~(\ref{eqn:parametrization}), using $\R$ and in each
  completion $B_{\nu_i}$ assigning the types by using the vertices in
  a fixed chamber containing $\R_{\nu_i}$ in the $SL_p(K_\nu)$
  building.  Thus every maximal order is isomorphic to exactly one
  order $\D^\gamma$, for $\gamma \in (\Z/p\Z)^m$. We have $\R =
  \D^{(0)}$.  Let $\gamma$ be fixed with $\E \cong \D^\gamma$.  To
  establish our claim, we need only show that $\Omega \subset
  D^\gamma$. By Proposition~\ref{prop:parametrize}, $\rho(\R,\E) =
  \rho(\R, \D^\gamma)$ so $\Frob_{L/K}(\rho(\R, \D^\gamma)) = 1$.
  Recall that $\D_\nu^\gamma = R_\nu$ for all $\nu \ne \nu_i$ and the
  primes $\nu_2, \dots, \nu_m$ all split completely in $L$. Thus
  $\Frob_{L/K}(\rho(\R, \D^\gamma)) =
  \Frob_{L/K}(\nu_1^{td_{\nu_1}(\R_{\nu_1}, \D^\gamma_{\nu_1})}) = 1$.
  Since $\Frob_{L/K}$ has order $p$ in $Gal(L/K)$, we have that
  $td_{\nu_1}(\R_{\nu_1}, \D^\gamma_{\nu_1}) \equiv 0 \pmod p$.  But
  given that the parametrization used the vertices in a fixed chamber
  of the building, this is only possible if $D^\gamma_{\nu_1} =
  R_{\nu_1}$, so of course $\Omega \subset D^\gamma_{\nu_1}$.  That
  $\Omega \subset D^\gamma_{\nu_i}$ for $i = 2, \dots, m$ follows in
  exactly the same way as in Proposition~\ref{prop:Omega-1}.  Finally
  for $\nu \ne \nu_i$, $\Omega \subset \R_\nu = \D_\nu^\gamma$.  Thus
  $\Omega \subset \D_\nu^\gamma$ for all primes $\nu$, and the
  argument is complete.
\end{proof}


\bibliographystyle{amsplain} 

\providecommand{\bysame}{\leavevmode\hbox to3em{\hrulefill}\thinspace}
\providecommand{\MR}{\relax\ifhmode\unskip\space\fi MR }
\providecommand{\MRhref}[2]{%
  \href{http://www.ams.org/mathscinet-getitem?mr=#1}{#2} }
\providecommand{\href}[2]{#2}

\end{document}